 \documentclass[12pt]{amsart}                                 
                                                                                                             
\usepackage{amscd,amsthm,amsmath,amssymb,amsfonts}

\newtheorem{Theorem}{Theorem}[section]

\newtheorem{Lemma}[Theorem]{Lemma}

\def\buildreb#1\over#2{\mathrel{\mathop{\kern0pt #1}\limits_{#2}}}

\def\fin{\end{document}}

\usepackage{nicefrac,xcolor,upref}

\newtoks\baspagetitre
\baspagetitre={\hfill{}}

\newif\ifpagetitre \pagetitretrue 
\def\footnoterule{\kern -3pt\hrule width 2truein \kern 2.4pt}
\renewcommand{\thepage}{\ifpagetitre\the\baspagetitre\global\pagetitrefalse
 \else \arabic{page}\fi} 
 
 \def\eps{\varepsilon} \def\alp{\alpha} \def\Z{{\mathbb Z}}  \def\R{{\mathbb R}}
 \def\Q{{\mathbb Q}} \def\gam{\gamma} \def\sig{\sigma}   \def\C{{\mathbb C}}
 
  \def\Om{{\bf {\Omega}}} \def\bsl{\backslash} \def\tr{{\rm Tr}} \def\om{\omega} 
    \def\Gam{\Gamma}
 
\def\rme{{\rm e}}
\def\rmi{{\rm i}} 

\begin{document}

\title[Refinements of Peck's theorem]{Refinements of Peck's theorem on simultaneous approximation to algebraic numbers}

\author{Yann Bugeaud}
\address{Universit\'e de Strasbourg, Math\'ematiques, 7, rue Ren\'e Descartes, 67084 Strasbourg  (France)} 
\address{Institut universitaire de France} 
\email{bugeaud@math.unistra.fr}

\author{Bernard de Mathan} 
\email{bernard.demathan@gmail.com}

\begin{abstract}
Let $n$ be an integer with $n\ge2$, and let $E$ be a real algebraic number field of degree $n+1$ over $\Q$. 
Let $\alp_1, \ldots , \alp_n$ be real numbers in $E$ such that $(1,\alp_1,\ldots,\alp_n)$ is a linear basis of $E$ over $\Q$.  
Let $\nu_1,  \ldots, \nu_{n-1}$ be real numbers satisfying 
 $$0<\max_{1\le i\le n-1}\nu_i\le2\min_{1\le i\le n-1}\nu_i,\quad \nu_1+ \ldots +\nu_{n-1}=1. $$
 We establish that there exist a real number $C$, depending only on $\alpha_1, \ldots , \alpha_n$, and infinitely many 
 integers $Q \ge 2$ satisfying
the inequalities 
$$Q^{1/n}\Vert Q\alp_i\Vert \le C  (\log Q)^{-\nu_i}, \quad 1\le i\le n-1,
\quad 
Q^{1/n}\Vert Q\alp_n\Vert\le C.$$
This answers partially a conjecture of Peck, who proved in 1961 this statement in the particular case
where $\nu_1 = \ldots = \nu_{n-1} = 1/ (n-1)$. 
We also improve a result from 2004 of de Mathan and Teuli\'e on a question of simultaneous Diophantine 
approximation involving a non-Archimedean valuation. 
\end{abstract}


\maketitle

\leftskip = 0pt
\rightskip = 2pt

\section{Introduction.} 
The Littlewood conjecture in Diophantine approximation  
asserts that, for any pair $(\alp_1,\alp_2)$ of real numbers, we have  
$$
\inf_{Q \in\Z\bsl\{0\}} \, |Q| \cdot \Vert Q\alp_1\Vert \cdot \Vert Q\alp_2\Vert=0,\eqno\hbox{(Lit)}
$$
where $\Vert .\Vert$ denotes the distance to the nearest integer.
A weaker conjecture is still open, which states that there exists an integer $n \ge 2$ such that, for any $n$-tuple $(\alp_1,\ldots,\alp_n)$  of real numbers, we have  
$$\inf_{Q \in\Z\bsl\{0\}} \, |Q| \cdot \Vert Q\alp_1\Vert\cdots\Vert Q\alp_n\Vert=0. \eqno\hbox{(Litn)}$$
Only few examples of tuples satisfying (Litn) non-trivially are known.
In 1955, Cassels and Swinnerton-Dyer \cite{CS} proved that (Lit) is true for every pair $(\alp_1,\alp_2)$ of real numbers in a cubic field.
Subsequently, Peck \cite{Pe} extended their result in proving  
that (Litn) holds when $\alp_1,\ldots,\alp_n$ are real numbers 
in an algebraic number field of degree $n+1$. 

More precisely, when $1, \alpha_1, \ldots , \alpha_n$ are linearly independent over $\Z$, 
Peck established that there are infinitely many integers $Q \ge 2$ satisfying 
$$Q^{1/n}\Vert Q\alp_i\Vert\ll{1\over(\log Q)^{1/(n-1)}}, \quad 1\le i\le n-1,$$ and
$$Q^{1/n}\Vert Q\alp_n\Vert\ll1,$$
thus,
$$
\liminf_{Q \to + \infty} \,  Q \, (\log Q) \, \Vert Q\alp_1\Vert \cdots \Vert Q\alp_n\Vert<+\infty, 
$$
while we have   
$$
\liminf_{Q \to + \infty}  \, Q^{1/n}\max\{\Vert Q\alp_1\Vert,\ldots,\Vert Q\alp_n\Vert\} > 0. 
$$

Throughout the paper, we use the Vinogradov notation. 
For positive quantities $A$ and $B$, we write $A\ll B$ when $A\le CB$ for a positive constant $C$. 
Furthermore, $A\gg B$  means that $B\ll A$, and $A\asymp B$ means that we have 
both $A\ll B$ and $B\ll A$. The constants implicit in Peck's result depend only on the numbers 
$\alp_1,\ldots,\alp_n$. Actually, all the implicit constants occurring in this paper depend only on
$\alp_1,\ldots,\alp_n$ and (in Theorems 1.2 and 1.3 and in their proofs) on a prime number $p$. 

After stating his result, Peck writes that `one might conjecture that we can find infinitely many solutions of
the inequalities'
$$Q^{1/n}\Vert Q\alp_i\Vert\ll{1\over f_i (Q)}, \quad 1\le i\le n-1,
\quad 
Q^{1/n}\Vert Q\alp_n\Vert\ll1,   \eqno\hbox{(1.1)}$$
with $f_1 (Q) \cdots f_{n-1} (Q) = \log Q$ and $f_i (Q) \ge 1$ for $i = 1, \ldots , n-1$. 

Our first result allows us to confirm Peck's conjecture 
under an additional assumption on $f_1, \ldots , f_{n-1}$.

\begin{Theorem} 
Let $n$ be an integer with $n\ge2$, and let $E$ be a real algebraic number field of degree $n+1$ over $\Q$. 
Let $\alp_1, \ldots , \alp_n$ be real numbers in $E$ such that $(1,\alp_1,\ldots,\alp_n)$ is a linear basis of $E$ over $\Q$.  
There exists a positive constant $C$, depending only on $\alp_1, \ldots , \alp_n$, such that, given 
any real numbers $\eps_1, \ldots , \eps_{n-1}$ with
$$0<\max_{1\le i\le n-1}\eps_i^2\le\min_{1\le i\le n-1}\eps_i,\quad1\le i\le n-1, \eqno\hbox{\rm(1.2)}$$
there exists a positive integer $Q$ such that
$$Q^{1/n}\Vert Q\alp_i\Vert\le C\eps_i,\quad1\le i\le n-1,  \quad   Q^{1/n}\Vert Q\alp_n\Vert\le C, \eqno\hbox{\rm(1.3)}$$  
and $$ 
 \log Q\le {C\over\eps_1 \cdots \eps_{n-1}}\cdot\eqno\hbox{\rm(1.4)}$$ 
 \end{Theorem} 
 
Theorem 1.1 implies that, if we consider real numbers $\nu_1,  \ldots, \nu_{n-1}$ with
 $$0<\max_{1\le i\le n-1}\nu_i\le2\min_{1\le i\le n-1}\nu_i,\quad \nu_1+ \ldots +\nu_{n-1}=1,    $$
then Peck's conjecture holds for the functions $f_i$ 
defined by $f_i(Q)=(\log Q)^{\nu_i}$ for $i=1, \ldots , n-1$ and $Q \ge 2$. 
More generally, if the $f_i$ are non-decreasing functions with $f_i\ge1$,  $f_1(Q) \ldots f_{n-1}(Q)=\log Q$ for $Q \ge 2$, and 
$$
\max_{1\le i\le n-1} f_i (Q) \le 2 \min_{1\le i\le n-1}f_i (Q), \quad \hbox{for $Q$ sufficiently large,}  
$$ 
then (1.1) holds for infinitely many integers $Q \ge 2$. 
Indeed, if $\widehat Q$ is a positive integer, then, by (1.3) and (1.4), 
there is a positive integer $Q$ such that $$Q^{1/n}\Vert Q\alp_i\Vert\le {C^{n/(n-1)}\over f_i(\widehat Q)},\quad1\le i\le n-1,\quad Q^{1/n}\Vert Q\alp_n\Vert\le C,
$$ 
with $0<Q\le\widehat Q$, hence
 $$Q^{1/n}\Vert Q\alp_i\Vert\le {C^{n/(n-1)}\over f_i(Q)},\quad1\le i\le n-1,\quad Q^{1/n}\Vert Q\alp_n\Vert\le C.$$
 However, we do not know whether Theorem 1.1 continues to hold if we omit condition (1.2).   
 
Note that, under the assumption of Theorem 1.1, the Schmidt Subspace Theorem implies that   
$$
\liminf_{Q \to + \infty} \,  Q^{1 + \eps} \,  \Vert Q\alp_1\Vert \cdots \Vert Q\alp_n\Vert = +\infty, 
$$
for every $\eps > 0$. 
Let us add that it is proved in \cite{dual4} that, given $C>0$, if we consider the set ${\mathcal Q}_C$ of integers $Q \ge 2$ satisfying $$Q^{1/n}\Vert Q\alp_ n\Vert\le C,$$
then there exist constants $\kappa=\kappa(C)$ and $K=K(C)>0$ such that 
$$Q^{1/n}\max_{1\le i \le n-1}\Vert Q\alp_ i \Vert\ge K(\log Q)^{-\kappa}, \quad Q \in {\mathcal Q}_C.$$

As in \cite{Pe}, we construct dominant units in $E$, that is, units  
$\omega>1$ 
all of whose Galois conjugates, except 
$\omega$, have a modulus $\asymp | \omega |^{-1/n}$. Instead of using the crude estimate 
$\log (1 + x) = O(x)$ in a neighborhood of the origin as in \cite[bottom of p. 200]{Pe}, we use the more precise approximation 
$\log (1 + x) = x + O(x^2)$. The $2$ in the exponent yields the value $2$ occurring in the condition (1.2). 

Let $p$ be a prime number. 
Let $| \cdot |_p$ denote the $p$-adic absolute value normalized such that $|p|_p = p^{-1}$. 
The same idea, applied to the $p$-adic logarithm function, allows us also to improve a result of 
 \cite{dMT} asserting that, under the assumption of Theorem 1.1, 
 there are infinitely many integers $Q \ge 2$ satisfying 
 $$ Q^{1/n}\max\{\Vert Q\alp_1\Vert,\ldots,\Vert Q\alp_n\Vert\}\ll1$$ 
and $$|Q|_p(\log Q)^{1/n}<+\infty,$$ 
thus, 
$$
\liminf_{Q \to + \infty}  \, Q^{1/n}\max\{\Vert Q\alp_1\Vert,\ldots,\Vert Q\alp_n\Vert\}|Q|_p(\log Q)^{1/n}    
<+\infty. \eqno\hbox{\rm(1.5)} 
$$

\begin{Theorem} 
Let $E$ be a real algebraic number field of degree $n+1$ over $\Q$, with $n\ge1$. 
Let $r_1$ be the number of real embeddings of $E$ 
and $2r_2$ the number of  complex nonreal embeddings of $E$. 
Let $p$ be a prime number. 
Let $\alp_1, \ldots ,\alp_n$ be real numbers in $E$ such that $(1,\alp_1,\ldots,\alp_n)$ is a linear basis of $E$ over $\Q$. 
Then we have  
$$\liminf_{Q \to + \infty}  \, Q^{1/n}\max\{\Vert Q\alp_1\Vert,\ldots,\Vert Q\alp_n\Vert\}|Q|_p(\log Q)^{2/(r+1)}<+\infty. $$
More precisely, there exist sequences of integers $(Q_{i,s})_{s \ge 0}$,  $0\le i\le n$,   
 with $Q_{0,s}>0$ for $s \ge 0$, such that
$$Q_{0,s}^{1/n}|Q_{0,s}\alp_i-Q_{i,s}|\ll1, \quad 1\le i\le n,  $$
$$\log Q_{0,s}\ll p^{(r+1)s}, \quad |Q_{0,s}|_p\ll p^{-2s}, \quad |Q_{i,s}|_p\ll p^{-s}, \quad 1\le i\le n-1, $$
for $s \ge 0$, 
where the constants involved depend only on $\alp_1, \ldots , \alpha_n$, and $p$.    
\end{Theorem}

In \cite[Theorem 1.3]{Bg}, Bugeaud claims that

$$
\liminf_{Q \to + \infty} \, Q^{1/n}\max\{\Vert Q\alp_1\Vert,\ldots,\Vert Q\alp_n\Vert\}|Q|_p\log Q<+\infty,  $$ 
but his argument is incomplete. Namely, the constant $C_7$ on \cite[p. 79]{Bg} depends on $T$ and a correct 
computation yields only (with the notation of the present paper) the factor $(\log Q)^{1/n}$ as in (1.5). 
(actually,  we can get $(\log Q)^{1/r}$ by being more   
careful in \cite{Bg}). 

Note that, under the assumption of Theorem 1.2, the Schmidt Subspace Theorem implies that 
$$
\liminf_{Q \to + \infty} \, Q^{\eps + 1/n}\max\{\Vert Q\alp_1\Vert,\ldots,\Vert Q\alp_n\Vert\}|Q|_p = +\infty, 
$$
for every $\eps > 0$.  
We do not know whether Theorem 1.2 is optimal, however, we establish an 
analogue 
 of the result of 
\cite{dual4} mentioned above. 

\begin{Theorem} 
Under the assumptions of Theorem 1.2, for any positive real number $C$,     
there exist positive constants $\rho=\rho(C)$ and $K=K(C)$  such that for every $(n+1)$-tuple $(Q_0, Q_1, \ldots , Q_n)$ 
 of integers with  $Q_0 \ge 2$ and satisfying
$$Q_0^{1/n}|Q_0\alp_i-Q_i|\le C,  \quad  1\le i\le n,$$
we have
$$\max_{0\le i\le n-1}|Q_i|_p\ge K(\log Q_0)^{-\rho}.$$
\end{Theorem}

The present paper is organized as follows.
Theorems 1.1, 1.2, and 1.3 are established in Sections 2, 3, and 4, respectively. 
The final Section 5 is devoted to a short discussion on related questions.


\section{Proof of Theorem 1.1} 

Without any loss of generality, we assume that 
$$\eps_1 = \max_{1\le i\le n-1}\eps_i.   $$ 
Actually, we give an additional precision. We prove that there exists a positive constant $c$, depending 
only upon the numbers $\alp_1, \ldots ,\alp_n$, such that, given real numbers $\eps_1, \ldots, \eps_n$ 
with 
$$
\eps_1 = \max_{1\le i\le n-1}\eps_i, \quad 0<  \eps_1^2 \le \min_{1\le i\le n-1}\eps_i,  \eqno\hbox{(2.1)}  
$$ 
there exists an integer $Q \ge 2$ 
with $$\eps_1^{-c}\ll \log Q\ll {1\over\eps_1 \ldots \eps_{n-1}}\ ,\eqno\hbox{(2.2)}$$ such that 
$$
Q^{1/n}\Vert Q\alp_ i \Vert\ll \eps_ i  ,  \quad 1\le i \le n-1,\eqno\hbox{(2.2')}
$$ 
$$Q^{1/n}\Vert Q\alp_n\Vert\ll 1.\eqno\hbox{(2.2'')}
$$ 
There is no loss of generality in assuming that $\eps_1$ is small (in terms of  $\alp_1, \ldots ,\alp_n$). 
All the numerical constants $C_1, C_2, \ldots$ and those implicit in $\ll$ and $\gg$ depend only on $\alp_1, \ldots ,\alp_n$. 

The complex embeddings of $E$ are numbered as follows.
The embedding $\sig_0$ is the identity, 
$\sig_j$ is real for $1\le j\le r_1-1$, and $\sig_j$ is complex nonreal
 for $r_1\le j\le n$ with
$\sig_{j+r_2}=\overline{\sig_j}$ for $r_1\le j\le r_1+r_2-1$.
Recall that the group of units in $E$ has rank $r=r_1+r_2-1$. 
We shall use a set of $r$  multiplicatively independent units $\om_1, \ldots , \om_r$. 
By replacing if necessary $\om_l$ by $\om_l^2$,
we can suppose that for $l = 1, \ldots , r$ and $j = 0, \ldots , r_1 - 1$, 
we have $\sig_j(\om_l)>0$ 
(hence, in particular, $\om_l>0$). 
Set $\alpha_0 = 1$ and let $(\beta_ i )_{0\le i\le n}$ denote the dual basis of the basis 
$(\alp_ i )_{0\le  i \le n}$ of $E$. Consequently, we have 
$$\tr(\alp_ i \beta_j)=\delta_{ i ,j},\quad 0\le i,  j\le n, \eqno\hbox{(2.3)}$$ where $\delta_{ i ,j}=0$ if $ i \ne j$, and $\delta_{ i , i }=1$. 
Here and below, $\tr$ denotes the trace function defined by  
$$\tr=\sum_{0\le j\le n}\sig_j.$$
Let $D$ be a positive integer such that $D\alp_ i \beta_j$ is an algebraic integer in $E$ for every $i, j$ with $0\le  i , j \le n$. 
Let $w\ne0$ 
be an algebraic integer in $E$ and set 
$$Q_ i (w)=\tr(D\alp_ i \beta_nw),  \quad  0\le i \le n.\eqno\hbox{(2.4)}$$ 
The real numbers $Q_0, Q_1, \ldots , Q_n$  
are rational integers for $0\le i \le n$. We have 
$$Q_0(w)\alp_ i -Q_ i (w)=D\sum_{1\le j\le n}\sig_j(\beta_nw)(\alp_ i -\sig_j(\alp_ i )), \quad  1\le i \le n.\eqno\hbox{(2.4')}$$ 
It is important to note that, by (2.3), we have $$Q_ i (1)=0, \quad  0\le i \le n-1.$$ 
Thus, if, for $1\le i \le n$, we write 
$$Q_0(w)\alp_ i -Q_ i (w)=D\sig_1(w)\sum_{1\le j\le n}\sig_j(\beta_n)(\alp_ i -\sig_j(\alp_ i )){\sig_j(w)\over\sig_1(w)}\,,\eqno\hbox{(2.5)}$$ then we see that
$$\sum_{1\le j\le n}\sig_j(\beta_n)(\alp_ i -\sig_j(\alp_ i ))=0,
 \quad 1\le i \le n-1,$$
and we get, for $1\le i \le n-1$, 
$$Q_0(w)\alp_ i -Q_ i (w)=D\sig_1(w)\sum_{2\le j\le n}\sig_j(\beta_n)(\alp_ i -\sig_j(\alp_ i ))\Bigl({\sig_j(w)\over\sig_1(w)}-1\Bigr).\eqno\hbox{(2.5')}$$
We will search the integer $Q$ of Theorem 1.1 of the form $Q_0(\om)$, where $\om$ is a suitable unit.
 Formula (2.5') invites us  
to replace  
the term $\sig_j(\om)/\sig_1(\om)-1$ by $\log(\sig_j(\om)/\sig_1(\om))$, which is easier to compute in terms of units.  

First we need the following lemma.

\begin{Lemma}
The $(n-1)\times(n-1)$-matrix
$$
{ \mathcal A}=
\begin{pmatrix}
\sig_2(\alp_1)-\alp_1  & \ldots &\sig_n(\alp_1)-\alp_1 \\ 
\ldots & \ldots  &\ldots \\ 
\sig_2(\alp_{n-1})-\alp_{n-1}  &  \ldots  &\sig_n(\alp_{n-1})-\alp_{n-1} 
\end{pmatrix}
$$
is invertible.  
\end{Lemma}

\begin{proof} 
Let us define the matrix 
$$
{ \mathcal B}=
\begin{pmatrix}
\sig_2(\beta_1)-{\sig_1(\beta_1)\over\sig_1(\beta_n)}\sig_2(\beta_n) &
  \ldots   & \sig_2(\beta_{n-1})-{\sig_1(\beta_{n-1})\over\sig_1(\beta_n)}\sig_2(\beta_n)  \\  \ldots &  \ldots &  \ldots 
    \\ \sig_n(\beta_1)-{\sig_1(\beta_1)\over\sig_1(\beta_n)}\sig_n(\beta_n)&
  \ldots   & \sig_n(\beta_{n-1})-{\sig_1(\beta_{n-1})\over\sig_1(\beta_n)}\sig_n(\beta_n) 
 \end{pmatrix} . 
$$
We will check that ${ \mathcal AB}$ is the identity matrix. 
We have to prove that 
$$\sum_{2\le k\le n}(\sig_k(\alp_ i )-\alp_ i )\Bigl(\sig_k(\beta_j)-{\sig_1(\beta_j)\over\sig_1(\beta_n)}\sig_k(\beta_n)\Bigr)=\delta_{ i ,j}, \quad 1 \le i, j \le n-1. \eqno\hbox{(2.6)}$$We have
$$\sum_{2\le k\le n}(\sig_k(\alp_ i )-\alp_ i )\Bigl(\sig_k(\beta_j)-{\sig_1(\beta_j)\over\sig_1(\beta_n)}\sig_k(\beta_n)\Bigr)=\sum_{0\le k\le n}(\sig_k(\alp_ i )-\alp_ i )\Bigl(\sig_k(\beta_j)-{\sig_1(\beta_j)\over\sig_1(\beta_n)}\sig_k(\beta_n)\Bigr),$$
since the additional terms in the sum on the right hand side are zero. 
Note that $$\sum_{0\le k\le n}(\sig_k(\alp_ i )-\alp_ i )\Bigl(\sig_k(\beta_j)-{\sig_1(\beta_j)\over\sig_1(\beta_n)}\sig_k(\beta_n)\Bigr)=\tr(\alp_ i \beta_j)-\alp_ i \tr(\beta_j)-{\sig_1(\beta_j)\over\sig_1(\beta_n)}(\tr(\alp_ i \beta_n)-\alp_ i \tr(\beta_n)).$$
By (2.3), we see that $\tr(\beta_j)=\tr(\alp_ i \beta_n)=0$ for $1\le i \le n-1$ and $1\le j\le n$, thus, (2.6) is proved. 
Actually, Lemma 2.1 is similar to Lemmas 2.3 and 2.4 in \cite{dual4}. 
\end{proof} 

Let $\eps_1, \ldots , \eps_{n-1}$ be real numbers satisfying (2.1), thus, with 
$$\eps_1 = \max_{1\le i\le n-1}\eps_i.$$
We construct a unit 
$$\om=\om_1^{m_1} \ldots \om_r^{m_r},$$ 
where $(m_1,\ldots,m_r)$ is in $\Z^r\bsl\{(0,\ldots,0)\}$
and satisfies 
$$\Bigl\vert\sum_{2\le j\le n}\sig_j(\beta_n)(\alp_ i -\sig_j(\alp_ i ))\Bigl({\sig_j(\om)\over\sig_1(\om)}-1 \Bigr)\Bigr\vert\ll \eps_ i ,  \quad 1\le i \le n-1.\eqno\hbox{(2.7)}$$ In order to obtain this condition, we use an analogue of the method 
introduced by Peck \cite{Pe}. 
We shall construct complex numbers $Y_j$ such that ${\sig_j(\om)\over\sig_1(\om)}= \rme^{Y_j}$, while requiring that $|Y_j|$ is small, in order to approach $ \rme^{Y_j}-1$ by $Y_j$, for $j= 2, \ldots , n$. 

We can avoid that $\sig_j(\om_l)$ is a negative real number for some $j$ with $r_1\le j\le n$, 
by replacing if necessary $\om_l$ by $\om_l^{2^\nu}$ where $\nu$ is a suitable positive integer. Clearly $\nu=r_2$ is convenient in all cases. 

For $j=1, \ldots , n$ and $l = 1, \ldots , r$, we define
 $\log\sig_j(\om_l)$ as being the  
 complex number such that 
 $$\rme^{\log\sig_j(\om_l)}=\sig_j(\om_l)
 $$ 
and whose argument is in $(- \pi, \pi]$. 
Recalling that  $\sig_j(\om_l)$ is a positive real number for every $j$ with $1 \le j < r_1$ and that
  $\sig_{j+r_2}(\om_l)=\overline{\sig_j(\om_l})$ for $r_1\le j<r_1+r_2$, we have
 $$\log\sig_j(\om_l)\in\R, \quad  1\le j<r_1,\eqno\hbox{(2.9)}$$
 and, since $\sig_j(\om_l)$ is never a negative real number,
 $$\log\sig_{j+r_2}(\om_l)=\overline{\log\sig_j(\om_l)}, \quad  r_1\le j<r_1+r_2.\eqno\hbox{(2.9')}$$

Let us set 
$$A_{ i ,j} = \sig_j(\beta_n)(\alp_ i -\sig_j(\alp_ i )),   \quad 0\le i,  j\le n.$$ 
Note that $A_{0,j}=0$ for $0\le j\le n$, and $A_{ i ,0}=0$ for $0\le i \le n$.
 
We construct a unit $\om$ satisfying (2.7) by applying the following lemma.

\begin{Lemma} 
There exist an $r$-tuple $(m_1,\ldots,m_r)$ of integers, not all zero, and an $n$-tuple
$(k_1,\ldots,k_n)$ in $\Z^n$, with
$$k_j=0\quad{\rm if}\quad1\le j<r_1, \quad  k_j+r_2=-k_j,\quad{\rm if}\quad r_1\le j\le r_1+r_2-1,\eqno\hbox{\rm(2.10)}$$
such that 
$$\Bigl\vert\sum_{2\le j\le n}A_{ i ,j}\sum_{1\le l\le r}m_l(\log\sig_j(\om_l)-\log\sig_1(\om_l))+2  \rmi \pi \sum_{2\le j\le n} A_{ i ,j}(k_j-k_1) \Bigr\vert\le \eps_ i , 
\quad  i =1,\ldots,n-1, \eqno\hbox{\rm(2.11)}$$
and $$\max_{1\le l\le r}|m_l|\ll \frac1{\eps_1\cdots\eps_{n-1}}. \eqno\hbox{\rm(2.12)}$$
\end{Lemma}

\begin{proof}
First, note that, for $l = 1, \ldots , r$,  we can write 
$$\sum_{2\le j\le n}A_{ i ,j}(\log\sig_j(\om_l)-\log\sig_1(\om_l))=\sum_{1\le j\le n}A_{ i ,j}(\log\sig_j(\om_l)-\log\sig_1(\om_l))=\sum_{1\le j\le n}A_{ i ,j}\log\sig_j(\om_l)$$ since, as $A_{ i ,0}=0$, we have
$$\sum_{1\le j\le n}A_{ i ,j}=\sum_{0\le j\le n}A_{ i ,j}=\alp_ i \tr(\beta_n)-\tr(\alp_ i \beta_n)=0, \quad 1\le i \le n-1.$$ Similarly, $$\sum_{2\le j\le n} A_{ i ,j}(k_j-k_1)=\sum_{1\le j\le n} A_{ i ,j}k_j, \quad 1\le i \le n-1.$$ 
Thus, condition (2.11) can be written
$$\bigl\vert\sum_{1\le j\le n}A_{ i ,j}\sum_{1\le l\le r}m_l\log\sig_j(\om_l)+2  \rmi \pi \sum_{1\le j\le n} A_{ i ,j}k_j \bigr\vert\le \eps_ i ,  \quad 1\le i \le n-1.\eqno\hbox{\rm(2.11')}$$
Note that $\sum_{1\le j\le n}A_{ i ,j}\log\sig_j(\om_l)$ is a real number, since, if $1\le j<r_1$, then $A_{ i ,j}$ and $\log\sig_j(\om_l)$ are real numbers (recall that $\sig_j(\om_l)$ is a positive real number), and if $r_1\le j<r_2$, then $$A_{ i ,j+r_2}=\overline{A_{ i ,j}}, \quad  \log\sig_{j+r_2} (\om_l)=\overline{\log\sig_j(\om_l)}.$$ Similarly, 
for an $r$-tuple $(k_j)_{1\le j\le n}$ satisfying (2.10), the complex number $2  \rmi \pi \sum_{1\le j\le n} A_{ i ,j}k_j$ is real, since
 $$\sum_{1\le j\le n} A_{ i ,j}k_j=\sum_{r_1\le j<r_1+r_2} (A_{ i ,j}-\overline{A_{ i ,j}})k_j
 = 2 \rmi \Im \sum_{r_1\le j<r_1+r_2} A_{ i ,j}k_j.$$ 
Let $M$ be a positive integer. 
The number of $n$-tuples $(m_1,\ldots,m_r,k_{r_1},\ldots,k_{r_1+r_2-1})$ in $\Z^n$ with $0\le m_l\le M$ and $0\le k_j\le M$ 
for $1\le l\le r$ and $r_1\le j\le r_1+r_2-1$ is $(M+1)^n$, thus greater than $M^n$. 
For such an $n$-tuple, let us set $k_j=0$ for $1\le j<r_1$ and $k_{j+r_2}=-k_j$ for $r_1\le j\le r_1+r_2-1$. 
Thus we get an $n$-tuple $(k_1,\ldots,k_n)$ in $\Z^n$ which satisfies (2.10). Then consider the point $X=(X_ i)_{1\le i \le n-1}$ in $\R^{n-1}$, where
 $$X_ i =\sum_{1\le j\le n}A_{ i ,j} \bigl(\sum_{1\le l\le r}m_l\log\sig_j(\om_l)+2  \rmi \pi k_j \bigr), \quad  1\le i \le n-1.$$ Let $C_1$ be an integer such that $$\max_{1\le i \le n-1}\Bigl(\sum_{1\le j\le n}\sum_{1\le l\le r}|A_{ i ,j}|| \log\sig_j(\om_l)|+2\pi\sum_{1\le j\le n}|A_{ i ,j}|\Bigr)\le C_1.$$
 Then the point $X$ belongs to $[-C_1M,C_1M]^{n-1}$. As we can suppose that $1/\eps_ i $ is an integer for each $ i $, this set can be covered by $(2C_1M)^{n-1}/ (\eps_ 1 \cdots \eps_{n-1})$ sets of the form $\prod_{1\le i \le n-1}[a_ i ,a_ i +\eps_ i ]$. Accordingly, if we take $$M={(2C_1)^{n-1}\over \eps_ 1 \cdots \eps_{n-1} },$$ then the pigeon-hole principle shows, by difference, that there exists $(m_1,\ldots,m_r,k_{r_1},  \ldots,k_{r_1+r_2-1})$ in $\Z^n$, different from $(0,\ldots,0)$, with
 $$\max_{1\le l\le r}|m_l|\le M, \quad  \max_{r_1\le j\le r_1+r_2-1}|k_j|\le M,$$ for which we have
 $$|X_ i |\le\eps_ i, \quad  i =1,\ldots,n-1. $$
 Thus we have proved (2.11') and (2.12), together with (2.10). We still must prove
that if $\eps_1$ is sufficiently small, then necessarily $(m_1,\ldots,m_r)$ is not equal to $(0,\ldots,0)$. 
For this, suppose that $m_l=0$ for each $l$ with $1\le l\le r$. Since (2.11) is equivalent to (2.11'), we have then
 $$\Bigl|\sum_{2\le j\le n} A_{ i ,j}(k_j-k_1) \Bigr\vert\le{\eps_1\over2\pi}, \quad 1\le i \le n-1. \eqno\hbox{(2.13)}$$ 
 Now, by Lemma 2.1, the matrix $(A_{ i ,j})_{1\le i \le n-1,\,2\le j\le n}$  
 is invertible, thus there exists a constant $C_2$ for which (2.13) implies that
 $$\max_{2\le j\le n}|k_j-k_1|\le C_2\eps_1.$$ Thus if $\eps_1<1/C_2$, we get $|k_j-k_1|<1$, hence $k_j-k_1=0$, for $2\le j\le n$. If $r_1>1$, we have $k_1=0$ by (2.10), and then $k_j=0$ for each $2\le j\le n$. This conclusion holds if $r_1=1$, since in this case, we have $k_{1+r_2}=k_1$, while by (2.10), $k_{1+r_2}=-k_1$, thus we also get $k_1=k_j=0$ for $2\le j\le n$. But this is impossible since we would have $(m_1,\ldots,m_r,k_{r_1},\ldots,k_{r_1+r_2-1})=(0,\ldots,0,\ldots0)$.
 \end{proof}

We are in position to complete the proof of the theorem.  
Let $(m_1,\ldots,m_r,k_{r_1},\ldots,k_{r_1+r_2-1})$ be an $n$-tuple satisfying Lemma 2.2, with $(m_1,\ldots,m_r)\ne(0,\ldots,0)$, and consider
$$\om=\om_1^{m_1}.\ldots \om_r^{m_r}.\eqno\hbox{(2.14)}$$ 
This number is a unit in $E$, and $\om\ne1$ since $\om_1, \ldots ,\om_r$ are multiplicatively independent. As we can replace the $n$-tuple $(m_1,\ldots,m_r,k_{r_1},\ldots,k_{r_1+r_2-1})$ by its opposite, we can suppose that $\om>1$. Then, consider the complex numbers
$$Y_j=\sum_{1\le l\le r}  \bigl( m_l(\log\sig_j(\om_l)-\log\sig_1(\om_l))+2  \rmi \pi (k_j-k_1) \bigr), \quad  2\le j\le n.\eqno\hbox{(2.15)}$$
We have 
$$ \rme^{Y_j}={\sig_j(\om)\over\sig_1(\om)}, \quad 2 \le j \le n.   \eqno\hbox{(2.15')}$$
Condition (2.11) means that 
$$\Bigl|\sum_{2\le j\le n}A_{ i ,j}Y_j  \Bigr|\le\eps_ i , \quad 1\le i \le n-1. \eqno\hbox{(2.16)}$$ 
As the matrix $(A_{ i ,j})_{1\le i \le n-1,\,2\le j\le n}$ is invertible, (2.16) implies that
$$|Y_j|\ll\eps_1, \quad 2\le j\le n.\eqno\hbox{(2.17)}$$ Now, by (2.15') and (2.17),
$${\sig_j(\om)\over\sig_1(\om)}=1+Y_j+O(\eps_1^2), \quad 2\le j\le n, 
\eqno\hbox{(2.18)}$$ (where the constant in the $O$ only depends upon $\alp_1,\ldots,\alp_n$).
Then, set
$$Q_ i =Q_ i (\om)=D\tr(\alp_ i \beta_n\om), \quad  0\le i \le n.\eqno\hbox{(2.4')}$$
Recalling that 
$$\eps_1^2\le\eps_ i , \quad 1\le i \le n-1,\eqno\hbox{(2.1')}$$ we deduce from (2.5'), (2.16), (2.17), and (2.18) that 
$$|Q_0(\om)\alp_ i -Q_ i (\om)|\ll|\sig_1(\om)|\eps_ i , \quad 1\le i \le n-1.\eqno\hbox{(2.19)}$$ 
As $0<\eps_1\le1$,   
(2.17) and (2.18) imply that $$\left\vert{\sig_j(\om)\over\sig_1(\om)}\right|\ll1, \quad 1\le j\le n, \eqno\hbox{(2.20)}$$ and thus, by (2.5), we obtain also
$$|Q_0(\om)\alp_n-Q_n(\om)|\ll|\sig_1(\om)|.\eqno\hbox{(2.19')}$$ 
If $\eps_1$ is sufficiently small (in terms of $\alp_1 , \ldots, \alpha_n$), 
then,  
 by (2.17) and (2.18), we have 
 $$\frac12\le\left|{\sig_j(\om)\over\sig_1(\om)}\right|\le\frac32, \quad 1 \le j \le n,$$ and,    
as $$\prod_{0\le j\le n}\sig_j(\om)=1,$$ we get 
$$|\sig_j(\om)|\asymp\om^{-1/n}, \quad  1\le j\le n.\eqno\hbox{(2.21)}$$
Moreover, if $\eps_1$ is small, then $\om$ is large. 
Indeed, 
by (2.15) and (2.17), 
there exists an integer $j$ with $2\le j\le n$ such that  
$$0<\Bigl|\sum_{1\le l\le r}m_l(\log\sig_j(\om_l)-\log\sig_1(\om_l))+2  \rmi \pi (k_j-k_1)\Bigr|\ll\eps_1. \eqno\hbox{(2.21')}$$
Otherwise,    
if $Y_j=0$ for every $j = 2, \ldots ,n$, we would get $\sig_j(\om)=\sig_1(\om)$ for $j=2,\ldots,n$, thus $\om$ would be a rational number, and $\om=1$, 
a contradiction.   
Thus, there is at least one integer $j$ with $2\le j\le n$ such that $0<|Y_j|\ll \eps_1$.     
Then, 
the set of possible integers $j$ being finite, a classical theorem of Baker \cite[Theorem 1.11]{Bu18} ensures that 
$$
\max\{|m_1|, \ldots , |m_r|, |k_j - k_1| \} \gg\eps_1^{-c},
$$ 
for a suitable positive constant $c$, which depends only upon $\alp_1,\ldots,\alp_n$. 
Since (2.21') implies that $|k_j - k_1| \ll \max\{|m_1|, \ldots , |m_r| \}$, we get 
$$
\max\{|m_1|, \ldots , |m_r|  \} \gg\eps_1^{-c}. 
$$
Then, consider the system equalities
$$m_1\log|\sig_j(\om_1)|+\ldots +m_r\log|\sig_j(\om_r)|=\log|\sig_j(\om)|, \quad  1\le j\le r.$$ 
Since the units $\om_1, \ldots , \om_r$ are independent, the matrix $(\log |\sigma_j(\om_i)|)_{1 \le i, j \le r}$ is invertible and we have 

$$\max_{1\le l\le r}|m_l|\asymp\max_{1\le j\le r}|\log|\sig_j(\om)||,$$ hence, in view of (2.21),
$$\log\om\asymp\max_{1\le l\le r}|m_l|\gg\eps_1^{-c}.\eqno\hbox{(2.22)}$$
Thus, $\om$ is large when $\eps_1$ is small in terms of $\alp_1,\ldots,\alp_n$, and if we write (2.4) explicitly: 
$$Q_0=D\Bigl(\beta_n\om+\sum_{1\le j\le n}\sig_j(\beta_n)\sig_j(\om)\Bigr),$$ then, by (2.21), we see that
 $$|Q_0|\asymp\om.$$ By (2.21), we then have
 $$|\sig_1(\om)|\asymp|Q_0|^{-1/n},$$ accordingly we get (2.2') and (2.2'') by (2.19) and (2.19'). Lastly, we deduce (2.2) from (2.22) and (2.12).  
 Thus Theorem 1.1 is proved.

 
\section{Proof of Theorem 1.2}
We keep the same notation as above.  
All the numerical constants $C_1, C_2, \ldots$ and those implicit in $\ll$ and $\gg$ 
depend only on $\alp_1, \ldots ,\alp_n$, and $p$.  
For $s \ge 0$ and $i=0, \ldots , n$, we shall find the integers $Q_{i,s}$ of the form
$$Q_{i,s}=D\tr(\alp_i\beta_nu),\eqno\hbox{(3.1)}$$ where $u=u(s)$ is a unit we will construct. 
Let $\tau_0, \tau_1, \ldots , \tau_n$ be the embeddings of $E$ in  the  
algebraic closure ${\bf\Omega}_p$ of $\Q_p$ ($\Om_p$ is equipped with the 
absolute value $| \cdot |_p$  extending the $p$-adic absolute value on $\Q_p$). Since the traces can be calculated in $\Om_p$ as well as in $\C$, we can write in $\C$
$$Q_{i,s}=D\sum_{0\le j\le n}\sig_j(\alp_i\beta_nu), \quad  0\le i\le n,\eqno\hbox{(3.2)}$$ 
and in $\Om_p$
$$Q_{i,s}=D\sum_{0\le j\le n}\tau_j(\alp_i\beta_nu),  \quad  0\le i\le n.\eqno\hbox{(3.3)}$$ 
As in (2.5), we write in $\C$
$$Q_{0,s}\alp_i-Q_{i,s}=D\sig_1(u)\sum_{1\le j\le n}\sig_j(\beta_n)(\alp_i-\sig_j(\alp_i)) 
\frac{\sig_j(u)}{\sig_1(u)}\,, \quad 1\le i\le n, \eqno\hbox{(3.4)}$$ while, in $\Om_p$,
as $\tr(\alp_i\beta_n)=0$ for $i = 0, \ldots , n-1$, we get 
$$Q_{i,s}=D\sum_{1\le j\le n}\tau_j(\alp_i\beta_n)(\tau_j(u)-\tau_0(u)), \quad  0\le i\le n-1.\eqno\hbox{(3.5)}$$ 
Thus, we have
$$Q_{i,s}=D\tau_0(u)\sum_{1\le j\le n}\tau_j(\alp_i\beta_n)\Bigl({\tau_j(u)\over\tau_0(u)}-1\Bigr), \quad  0\le i\le n-1.\eqno\hbox{(3.6)}$$ 

We shall use the $p$-adic logarithm function $\log_p$. For $z$ in $\Om_p$ with $|z|_p<p^{-1/(p-1)}$, this function is defined by
$$\log_p(1+z)=\sum_{k=1}^{+\infty}(-1)^{k-1}\frac{z^k}{k}.$$
It satisfies 
$$|\log_p(1+z)|_p=|z|_p,\eqno\hbox{(3.7)}$$ and, more precisely, 
$$|\log_p(1+z)-z|_p=\frac{|z|^2}{|2|_p}\cdot\eqno\hbox{(3.7')} $$
We use multiplicatively independent units $\om_1, \ldots ,\om_r$ in $E$ such that 
$$
|\tau_j(\om_l)-1|_p<p^{-1/(p-1)}, \quad 0\le j\le n,\,1\le l\le r.
$$
This is possible since, for $j=0, \ldots,\ n$, we have $|\tau_j(\om_l)|_p=1$, and, for every $z$ in $\Om_p$ with $|z|_p=1$, there is a positive integer $k$ such that $|z^k-1|_p<p^{-1/(p-1)}$ (clearly, an effective value of $k$, depending only upon the degree of $z$ over $\Q_p$, can be given); thus, it suffices to replace, if necessary, $\om_l$ by $\om_l^k$. 

Given an integer $s \ge 0$, we first construct a unit $\om$, with $0<|\om-1|_p<p^{-1/(p-1)}$, such that we have in $\C$ 
$$\left|\log\left|{\sig_j(\om)\over\sig_1(\om)}\right|\right|\le p^{-s}, \quad  2\le j\le n,\eqno\hbox{(3.8)}$$ and in $\Om_p$ 
$$\Bigl| \sum_{j=1}^n\tau_i(\beta_n)\log_p{\tau_j(\om)\over\tau_0(\om)}\Bigr|_p\le p^{-s}.\eqno\hbox{(3.9)}$$ 
 Note that $$\sum_{j=1}^n\tau_j(\beta_n)\log_p{\tau_j(\om)\over\tau_0(\om)}=
\sum_{j=0}^n\tau_j(\beta_n)(\log_p\tau_j(\om)-\log_p\tau_0(\om))=\sum_{j=0}^n\tau_j(\beta_n)\log_p\tau_j(\om),$$
since $\sum_{j=0}^n\tau_j(\beta_n)=0$. Thus (3.9) can also  be written
$$
\biggl|\sum_{j=0}^n\tau_j(\beta_n)\log_p\tau_j(\om) \biggr|_p\le p^{-s}.\eqno\hbox{(3.9')}$$
We search $\om$ of the form $\om=\om_1^{m_1}\cdots\om_r^{m_r}$, where the $m_l$ are integers, not all zero. We shall use the following lemma.

\begin{Lemma} There exist integers $m_1, \ldots ,m_r$, not all zero, such that 
$$\Bigl|\sum_{l=1}^rm_l\log\bigl|{\sig_j(\om_l)\over\sig_1(\om_1)}\bigr|\Bigr|\le p^{-s}, \quad  2\le j\le r,\eqno\hbox{\rm(3.10)}$$
$$\Bigl|\sum_{l=1}^rm_l\sum_{j=0}^n\tau_j(\beta_n)\log_p\tau_j(\om_l) \Bigr|_p\le p^{-s}, \eqno\hbox{\rm(3.11)}$$ and
$$\max_{1\le l\le r}|m_l|\ll p^{rs}.   \eqno\hbox{\rm(3.12)}$$  
\end{Lemma} 

\begin{proof}
First, note that for every $x$ in $ E$ such that $$|\tau_j(x)-1|_p<p^{-1/(p-1)}, \quad  0\le j\le n,\eqno\hbox{(3.13)}$$
the sum $\sum_{j=0}^n\tau_j(\beta_n)\log_p \tau_j(x)$ lies in $\Q_p$. Indeed, by (3.13), we have 
$$\log_p\tau_j(x)= \sum_{k=1}^{+\infty}(-1)^{k-1}{(\tau_j(x)-1)^k\over k}, \quad j=0,\ldots,n,$$ 
and, as we suppose that $x$ belongs to $E$,
$$\sum_{j=0}^n\tau_j(\beta_n)\log \tau_j(x)=\sum_{k=1}^{+\infty}(-1)^ 
{k-1}\sum_{j=0}^n\tau_j(\beta_n){\tau_j((x-1)^k)\over k}=\sum_{k=1}^{+\infty}(-1)^ 
{k-1}{\tr(\beta_n(x-1)^k)\over k}\cdot$$ 
This series is convergent in $\Om_p$, while all its terms are rational numbers. Hence its sum lies in $\Q_p$.

Then, for an $r$-tuple $(m_1,\ldots,m_r)$ in $\Z^r$, consider the point  $X=(X_j)_{2\le j\le r}$ in $\R^{r-1}$, where
$$X_j=\sum_{l=1}^rm_l\log\left|{\sig_j(\om_l)\over\sig_1(\om_1)}\right|, \quad  2\le j\le r,$$ and the point $Y$ in $\Q_p$
defined by 
$$Y=\sum_{l=1}^rm_l\sum_{j=0}^n\tau_j(\beta_n)\log_p\tau_j(\om_l).$$ 
Given a positive integer $M$, if we take integers $m_1, \ldots , m_r$ in $[0, M]$,  
then we get $(M+1)^r$ points $(m_1,\ldots,m_r)$ in $\Z^r$. Let $C_3$ be a positive integer such that $$\sum_{l=1}^r\left|\log\left|{\sig_j(\om_l)\over\sig_1(\om_1)}\right|\right|\le C_3, \quad  2\le j\le r,$$
and $b$ a nonnegative integer such that $$|\tau_j(\beta_n)|_p \, p^{-1/(p-1)}\le p^b, \quad  0\le j\le n.$$  
For any of the points $(X,Y)$, we have
$$|X_j|\le C_3M, \quad  |Y|_p\le p^b.$$
In other words, the pair $(X,Y)$ belongs to the set $[-C_3M,C_3M]^{r-1}\times p^{-b}Z_p.$ Now this set can be covered by $(2C_3)^{r-1}M^{r-1}p^{rs+b}$ sets of the form $(\prod_{2\le j\le r}[a_j,a_j+p^{-s}])\times 
 (y+p^s\Z_p)$.
By the pigeon-hole principle, if we set 
$$M=(2C_3)^{r-1}p^{rs+b},\eqno\hbox{(3.14)}$$
then, we have 
$$(2C_3)^{r-1}M^{r-1}p^{rs+b}=M^r<(M+1)^r, $$
and we get by difference a point $(m_1,  \ldots ,m_r)$ in $\Z^r\bsl\{0,\ldots,0\}$ 
satisfying (3.10) and (3.11) with $$\max_{1\le l\le r}|m_l|\le M.$$ Thus, in view of (3.14), Lemma 3.1 is proved.
\end{proof} 

We are now in position to complete the proof of Theorem 1.2.  
Let $(m_1,\ldots,m_r)$ be a nonzero $r$-tuple in $\Z^r$ satisfying the three inequalities in Lemma 3.1. Put
$$\om=\om_1^{m_1}\cdots\om_r^{m_r}.$$ 
Conditions (3.10) and (3.11) mean that the number $\om$ satisfies (3.8) and (3.9), and, as
before, $\om_1, \ldots ,\om_r$ being multiplicatively independent, we have $$\max_{1\le j\le n}|\log|\sig_j(\om)||\asymp\max_{1\le l\le r}|m_l|.$$ Hence
$$\max_{1\le j\le n}|\log|\sig_j(\om)||\gg1.\eqno\hbox{(3.15)}$$
Now, by (3.8), we have $$|\log|\sig_j(\om)|-\log|\sig_1(\om)||\le p^{-s}, \quad j=2, \ldots , n, \eqno\hbox{(3.8')}$$ thus, as we can suppose without loss of generality that $s$ is large, we deduce from (3.15) and (3.8') that  $$|\log|\sig_1(\om)||\gg1.\eqno\hbox{(3.16)}$$
Since $\omega$ is a unit, we have
$$\log|\om|+\sum_{1\le j\le n}\log|\sig_j(\om)|=0,$$
thus, by (3.8'),  $$|\log|\om|+n\log|\sig_1(\om)||\le(n-1)p^{-s},\eqno \hbox{(3.17)}$$
 and 
$$|\log|\om|+n\log|\sig_j(\om)||\le(2n-1)p^{-s}, \quad j=1,\,\ldots,\,n. \eqno \hbox{(3.17')}$$ Thus, (3.16) and (3.17) imply that, if $s$ is large, then we have
$$|\log|\om||\gg1.$$ Hence $|\om|\ne1$, and, since we may replace, if necessary, the $r$-tuple $(m_1, \ldots , m_r)$ 
by its opposite, we can suppose that $|\om|>1$. We have then $$\log|\om|\gg1.$$ 
It follows from (3.17') that $$|\sig_j(\om^{p^s})|\asymp|\om^{-p^s/n}|, \quad 1\le j\le n,\eqno\hbox{(3.18)}$$ while 
$$\log |\om|^{p^s} \gg p^s,\eqno\hbox{(3.19)}$$
 that is to say that $|\om^{p^s}|$ is large when $s$ is large.
Hence, as in Section 2, putting
$$Q_{i,s}=D\tr(\alp_i\beta_n\om^{p^s}), \quad i=0, \ldots , n,$$ we get by (3.18) and (3.19) that 
$$|Q_{0,s}|\asymp|\om|^{p^s},\eqno\hbox{(3.20)}$$ and, by (3.18) and (3.4) with $u=\om^{p^s}$,
$$|Q_{0,s}|^{1/n}|Q_{0,s}\alp_i-Q_{i,s}|\ll1,  \quad 1\le i\le n.$$
By (3.6), with $u=\om^{p^s}$, we have
$$|Q_{i,s}|_p\le\biggl|\sum_{1\le j\le n}\tau_j(D\alp_i\beta_n)\Bigl({\tau_j(\om^{p^s})\over\tau_0(\om^{p^s})}-1\Bigr)\biggr|_p, \quad  0\le i\le n-1.\eqno\hbox{(3.6')}$$
Now, by (3.7), $$\left|{\tau_j(\om^{p^s})\over\tau_0(\om^{p^s})}-1\right|_p=\left|\log_p{\tau_j(\om^{p^s})\over\tau_0(\om^{p^s})}\right|_p=p^{-s}\left|\log_p{\tau_j(\om)\over\tau_0(\om)}\right|_p<p^{-s-1/(p-1)}, \quad j = 1, \ldots , n.$$ 
Hence, as $D\alp_i\beta_n$ is an algebraic integer for $i=1, \ldots , n-1$, we get 
$$|Q_{i,s}|_p\le p^{-s}, \quad 1\le i\le n-1.\eqno\hbox{(3.5')}$$
This inequality holds for $i=0$, but we can give a more precise result in this case. Indeed, by (3.7'),
$$\left|{\tau_j(\om^{p^s})\over\tau_0(\om^{p^s})}-1-p^s
\log_p{\tau_j(\om)\over\tau_0(\om)}\right|_p\le \frac1{|2|_p}
\left|{\tau_j(\om^{p^s})\over\tau_0(\om^{p^s})}-1\right|_p^2\le \frac1{|2|_p} p^{-2s}p^{-2/(p-1)} \le p^{-2s}, 
\quad j=1, \ldots , n, 
$$ and as, by (3.9), 
$$\left|p^s\sum_{1\le j\le n}\tau_j(\beta_n)\log_p{\tau_j(\om)\over\tau_0(\om)}\right|_p\le p^{-2s},$$
we get by (3.6') 
$$|Q_{0,s}|_p\le p^{-2s}.\eqno\hbox{(3.5'')}$$ Lastly, as $\log |\om| \ll\max_{1\le l\le r}|m_l|$, it follows from (3.12) and (3.20) that 
$$\log|Q_{0,s}|\ll p^{(r+1)s}.\eqno\hbox{(3.5'")}$$
Then, by (3.5'), (3.5''), and (3.5'"), condition (3.5) is proved. 
Moreover, we have obtained infinitely many different   
integers $Q_{0,s}$ since $$\log|Q_{0,s}|\gg p^s.$$ 
The proof of Theorem 1.2 is complete.


\section{Proof of Theorem 1.3} 

We keep the notation of Section 3. 
We use multiplicatively independent units $\om_1, \ldots ,\om_r$ in $E$ satisfying the condition:
$$|\tau_j(\om_l)-1|_p<p^{-1/(p-1)}, \quad 1\le l\le r, \quad 0\le j\le n.$$
 First, we need the following lemma.
 
 \begin{Lemma} There exists a sequence of units $(\eta_m)_{m \ge 1}$ of the form
 $$\eta_m=\om_1^{\mu_{1,m}} \cdots \om_r^{\mu_{r,m}},\eqno\hbox{\rm(4.1)}$$ where $\mu_{l,m}$ is in $\Z$
 for $1 \le l \le r$ and $m \ge 1$, 
 such that, for $m \ge 1$, we have 
$$ |\eta_m|\asymp  \rme^m,\eqno\hbox{\rm(4.2)}$$
$$|\sig_j(\eta_m)|\asymp  \rme^{-m/n}, \quad 1\le j\le n,\eqno\hbox{\rm(4.3)}$$
$$\max_{1\le l\le r}|\mu_{l,m}|\asymp m.\eqno\hbox{\rm(4.4)}$$ 

\end{Lemma}

\begin{proof}
See  \cite{Pe} or \cite{dual4}. 
\end{proof}

\begin{Lemma} Let $C$ be a positive real number. Denote by ${\mathcal Q_C}$ the set of $(n+1)$-tuples $(q_0,\ldots,q_n)$ in $\Z^{n+1}$, with $q_0\ne0$, for which $$|q_0|^{1/n}\vert q_0\alp_i-q_i\vert\le C, \quad  1\le i\le n.\eqno\hbox{\rm(4.5)}$$ 
There exists a finite set $\Gam$ in $E$ such that, for every $(q_0,\ldots,q_n)$ in $ {\mathcal Q_C}$ with $|q_0|$ sufficiently large (in terms of $\alp_1, \ldots , \alp_n$, and $C$), there exist a positive integer $m$ and a real number $\gam$ in $\Gam$ such that  
$$q_i=\tr(\gam\eta_m\alp_i), \quad  0\le i\le n. $$   
\end{Lemma} 

This lemma is proved in \cite{dual4}. For convenience, we reproduce below its proof. 

\begin{proof}
Recall that $(\beta_ i )_{0\le i\le n}$ denotes the dual basis of the basis 
$(\alp_ i )_{0\le  i \le n}$ of $E$ and that $D$ is a positive integer such that $D\alp_ i \beta_j$ 
is an algebraic integer in $E$ for every $i, j$ with $0\le  i , j \le n$. 
First, for an $(n+1)$-tuple $(q_0,\ldots,q_n)$ as in the lemma, setting
 $$\xi=\sum_{j=0}^nq_j\beta_j, $$ 
we have
$$\tr(\xi\alp_i)=q_i, \quad  0\le i\le n, \eqno\hbox{(4.6)}$$ 
and $D\xi$ is an algebraic integer.  
Let $m$ be the positive integer such that $$ \rme^{m-1}\le|q_0|< \rme^m.\eqno\hbox{(4.7a)}$$
 Put
$$\gam={\xi\over \eta_m}\cdot\eqno\hbox{(4.7b)}$$ 
Then, $D\gam$ is an algebraic integer. By (4.6), we have
$$q_0\alp_i-q_i=\sum_{1\le j\le n}(\alp_i-\sig_j(\alp_i))\sig_j(\xi), \quad  0\le i\le n. \eqno\hbox{(4.6')}$$  
Now, the $(n\times n)$-matrix
$(\alp_i-\sig_j(\alp_i))_{1\le i, j \le n}$ is invertible, since
$$\begin{vmatrix} 
\sig_1(\alp_1)-\alp_1 &\ldots &\sig_1(\alp_n)-\alp_n  \\
\ldots &\ldots &\ldots    \\ \sig_n(\alp_1)-\alp_1 &\ldots &\sig_n(\alp_n)-\alp_n \end{vmatrix} 
=\begin{vmatrix}
1&\alp_1 &\ldots &\alp_n  \\1&\sig_1(\alp_1) &\ldots &\sig_1(\alp_n)  \\
\ldots &\ldots &\ldots  \\1 &\sig_n(\alp_1)&\ldots &\sig_n(\alp_n) \end{vmatrix} \not= 0.$$
Thus, by (4.6'), we have 
$$\max_{1\le i\le n}|q_0\alp_i-q_i|\asymp\max_{1\le j\le n}|\sig_j(\xi)|,\eqno\hbox{(4.8)}$$
hence $$|\sig_j(\xi)|\ll C|q_0|^{-1/n}, \quad 1\le j\le n.\eqno\hbox{(4.8')}$$ 
Recalling that 
$$q_0= \tr(\xi \alpha_0) = \tr(\xi) = \sum_{0\le i\le n}\sig_i(\xi),$$ 
we deduce from (4.8') 
that $$\frac{|q_0|}{2}\le|\xi|\le{3|q_0|\over2}, \eqno\hbox{(4.9)}$$
for $|q_0|$ large enough in terms of $\alp_1,\ldots,\alp_n$ and $C$. 
 Accordingly, by (4.2), (4.3), (4.7a), (4.7b), and (4.9), we have 
$$|\gam|\ll 1,\quad  |\sig_j(\gam)|\ll  C, \quad 1\le j\le n.\eqno\hbox{(4.10)}$$
Now, there are only finitely many  
real numbers $\gam$ in $ E$ satisfying (4.10) and such that $D\gam$ is an algebraic integer. 
Thus Lemma 4.2 is proved.
\end{proof}

In order to complete the proof of Theorem 1.3, it is enough, by Lemma 4.2, to prove that if $\gam$ is a real number in $E$, 
and if we consider the  integers $q_i=\tr(\gam\eta_m\alp_i)$, for $i=0, \ldots, n$, with 
$$|q_0\alp_i-q_i|\le C|q_0|^{-1/n}, \quad  1\le i\le n,$$ then there exists a positive constant $\rho=\rho(\gam)$, such that we have, when $|q_0|$ is large (in terms of $\alp_1, 
\ldots , \alp_n$, and $\gam$), 
$$
\max_{0\le i\le n-1}|q_i|_p\gg_{\gam}(\log|q_0|)^{-\rho}.
$$
Here and below, the subscript ${}_\gamma$ means that the implicit constant also depends on $\gamma$. 

\begin{Lemma} Let $\xi$ be a real number in $E$. Set
$$q_i=\tr(\xi\alp_i), \quad  0\le i\le n.$$ Then
$$\max_{0\le i\le n-1}|\tr(\xi\alp_i)|_p\asymp
\max_{1\le j\le n}\left|\tau_j(\xi)-{\tau_j(\beta_n)\over\tau_0(\beta_n)}\tau_0(\xi)\right|_p.
$$
\end{Lemma}

\begin{proof} 
Let us write $$\tr(\xi\alp_i)=\sum_{0\le j\le n}\tau_j(\xi)\tau_j(\alp_i),$$ while
$$\tr(\alp_i\beta_n) = 0=\sum_{0\le j\le n}\tau_j(\beta_n)\tau_j(\alp_i), \quad i=0, \ldots, n-1.$$ 
Thus we get
$$\tr(\xi\alp_i)=\sum_{1\le j\le n}\Bigl(\tau_j(\xi)-{\tau_j(\beta_n)\over\tau_0(\beta_n)}\tau_0(\xi)\Bigr)\tau_j(\alp_i), \quad  0\le i\le n-1.$$
Now the matrix $(\tau_j(\alp_i))_{0\le i \le n-1, 1 \le j\le n}$ is invertible since 
$$
\begin{vmatrix}
1&\ldots &1    \\ \tau_1(\alp_1) &\ldots &\tau_n(\alp_1)    \\ \ldots &\ldots &\ldots    \\ \tau_1(\alp_{n-1}) &\ldots &\tau_{n}(\alp_{n-1}) \end{vmatrix}
=\begin{vmatrix} \tau_2(\alp_1)-\tau_1(\alp_1) &\ldots &\tau_{n}(\alp_1)-\tau_1(\alp_1)    \\ \ldots &\ldots &\ldots 
   \\ \tau_2(\alp_{n-1})-\tau_1(\alp_{n-1})&\ldots &\tau_{n}(\alp_{n-1})-\tau_1(\alp_{n-1}) \end{vmatrix} \not= 0,$$  
by Lemma 2.1. Thus, Lemma 4.3 is proved.
\end{proof}

Then we have to prove Theorem 1.3 for $(n+1)$-tuples $(q_0,\ldots,q_n)$,
 with $|q_0|\ge2$,  
 of the form 
$q_i=\tr(\gam\eta_m\alp_i)$, $0 \le i \le n$, 
where $\gam$ is a fixed nonzero real number in $E$.
As $|\tau_0(\eta_m)|_p 
=1$, we have by Lemma 4.3, $$\max_{0\le i\le n-1}|q_i|_p\gg_{\gam}\max_{1\le j\le n}\left|{\tau_j(\eta_m)\over\tau_0(\eta_m)}-{\tau_j(\beta_n)\tau_0(\gam)\over\tau_0(\beta_n)\tau_j(\gam)}\right|_p.$$ As $(q_0,\ldots,q_{n-1})\ne(0,\ldots,0)$, it follows from Lemma 4.3 that there is at least one integer $j$ with $1\le j\le n$ such that 
$${\tau_j(\eta_m)\over\tau_0(\eta_m)}-{\tau_j(\beta_n)\tau_0(\gam)\over\tau_0(\beta_n)\tau_j(\gam)}\ne0.$$
We have $|\tau_j(\eta_m)/\tau_0(\eta_m)-1|_p<p^{-1/(p-1)}$, thus, if $|\tau_j(\beta_n)\tau_0(\gam)/(\tau_0(\beta_n)\tau_j(\gam))-1 
|_p\ge p^{-1/(p-1)}$, then we get  $$\left|{\tau_j(\eta_m)\over\tau_0(\eta_m)}-{\tau_j(\beta_n)\tau_0(\gam)\over\tau_0(\beta_n)\tau_j(\gam)}\right|_p\ge p^{-1/(p-1)}.$$ If $|\tau_j(\beta_n)\tau_0(\gam)/(\tau_0(\beta_n)\tau_j(\gam))-1 
|_p<p^{-1/(p-1)}$, the function $\log_p$ being isometrical,  
 we can write
$$
\left|{\tau_j(\eta_m)\over\tau_0(\eta_m)}-{\tau_j(\beta_n)\tau_0(\gam)\over\tau_0(\beta_n)\tau_j(\gam)}\right|_p=\left|\log_p{\tau_j(\eta_m)\over\tau_0(\eta_m)}-\log_p{\tau_j(\beta_n)\tau_0(\gam)\over\tau_0(\beta_n)\tau_j(\gam)}\right|_p,$$i.e.,
$$\left|{\tau_j(\eta_m)\over\tau_0(\eta_m)}-{\tau_j(\beta_n)\tau_0(\gam)\over\tau_0(\beta_n)\tau_j(\gam)}\right|_p=\left|\mu_{1,m}\log_p{\tau_j(\om_1)\over\tau_0(\om_1)}+\ldots+\mu_{r,m}\log_p{\tau_j(\om_r)\over\tau_0(\om_r)}-\log_p{\tau_j(\beta_n)\tau_0(\gam)\over\tau_0(\beta_n)\tau_j(\gam)}\right|_p. 
$$
By Kunrui Yu's $p$-adic version of Baker's Theorem (see \cite{BpY,Bu18}), 
we then get 
$$\left|{\tau_j(\eta_m)\over\tau_0(\eta_m)}-{\tau_j(\beta_n)\tau_0(\gam)\over\tau_0(\beta_n)\tau_j(\gam)}\right|_p
\gg_\gamma (\max_{1\le l\le r}|\mu_{l,m}|)^{-\rho},$$  
where $\rho$ is a positive constant (the constant involved in the symbol 
$\gg_\gamma$, as well as $\rho$, depends only on $\alp_1, \ldots ,\alp_n$, and $\gam$). As, by (4.4),  
$$
\max_{1\le l\le r}|\mu_{l,m}|\ll_\gam\log|q_0|, \quad m \ge 1, 
$$ 
we get
$$
\max_{0\le i\le n-1} |q_i|_p\gg(\log|q_0|)^{-\rho},
$$ 
where finally the involved constant and $\rho$ depend only on $\alp_1, \ldots ,\alp_n$, and $C$  
(indeed, when $C$ is given, we have to consider only a finite set of numbers $\gam$).

\section{Additional remarks}

For a basis $(1,\alp_1,\ldots,\alp_n)$ of a real number field $E$, 
it would be interesting to obtain other examples of integers $Q$ and positive numbers $\eps$, such that
$$|Q| \cdot \Vert Q\alp_1\Vert  \cdots \Vert Q\alp_n\Vert\le\eps.$$ 
If $n>2$, we are not able to prove that there exist integers $Q \ge 2$ such that
$$Q^{1/n}\Vert Q\alp_i\Vert\ll1, \quad 1\le i\le n-1,$$
and $$Q^{1/n}\Vert Q\alp_n\Vert\log Q\ll1.$$
Also, when $E$ is a cubic field, we do not know whether there exist pairs $(\lambda,\mu)$ of real numbers
other than $(1/2,1/2)$, with $0\le\lambda\le1$ and $\lambda+\mu=1$, and such that, for 
every $\eps > 0$,  there exist arbitrarily large positive integers $Q$ with
$$Q^\lambda \, \Vert Q\alp_1 \Vert \ll1 
\quad \hbox{and} \quad 
Q^\mu\, \Vert Q\alp_2\Vert \ll \eps.
$$

We may also consider linear forms in $1, \alpha_1, \ldots , \alpha_n$ and search for real numbers $\kappa > 0$, $C > 0$, and 
arbitrarily large integers $X$ with 
$$
|x_0 + x_1 \alpha_1 + \ldots + x_n \alpha_n| \le C X^{-n}, \quad |x_1|, \ldots , |x_n| \le CX,
$$
and
$$
|x_n| \le X (\log X)^{-\kappa}.
$$
The case $n=2$ is solved in \cite{BdM12}, where a weaker result is established when $n \ge 3$. 

In analogy with the conjecture of Littlewood, one can ask whether, given a positive integer $n$ and $(\alp_1,\ldots,\alp_n)$ in 
$\Q^n_p$, we have 
$$\inf_{\substack{q, r_1, \ldots , r_n \in \Z \\ qr_1 \cdots r_n\ne0}}|qr_1\cdots r_n||q\alp_1-r_1|_p \cdots |q\alp_n-r_n|_p=0.\eqno\hbox{(EKn)}$$   
When $n=1$, this problem is a particular case of a conjecture formulated by Einsiedler and Kleinbock  \cite{Eis}. 
If $\alp_1,\ldots,\alp_n$ lie in an algebraic number field of degree $n+1$, embedded in $\Q_p$, then (EKn) is true 
under a stronger form (see \cite{Teu} for $n \ge 1$  
and \cite{BdM19} for $n=1$). In the case  $n\ge2$, one can ask whether the results of \cite{Teu} may be refined. 

 
\medskip

\noindent 
{\it Note in press: } 
Alan Haynes informed us that, in a joint paper with Dhanda and Flynn, they proved
independently Theorem 1.1; see \cite[Theorem 5]{DFH26}.

\end{document}